\documentclass[12pt]{amsart}
\usepackage[utf8]{inputenc}
\usepackage{graphicx}
\usepackage{epstopdf}
\usepackage{inputenc}
\usepackage{enumitem}
\usepackage{amsmath,fullpage}
\usepackage{amssymb}
\usepackage{emptypage}
\usepackage{mathtools}
\usepackage[usenames,dvipsnames,x11names,svgnames]{xcolor}
\usepackage[colorlinks=true,linkcolor=NavyBlue,citecolor=DarkGreen, urlcolor=blue]{hyperref}
\usepackage{dsfont}

\newtheorem{theorem}{Theorem}
\newtheorem*{theorem*}{Theorem}

\newtheorem{corollary}{Corollary}[section]
\newtheorem{lemma}{Lemma}

\newtheorem*{acknowledgements*}{Acknowledgements}

\newcommand{\geqs}{\geqslant }

\makeatletter
\def\blfootnote{\gdef\@thefnmark{}\@footnotetext}
\makeatother

\def\house#1{\setbox1=\hbox{$\,#1\,$}%
\dimen1=\ht1 \advance\dimen1 by 2pt \dimen2=\dp1 \advance\dimen2 by 2pt
\setbox1=\hbox{\vrule height\dimen1 depth\dimen2\box1\vrule}%
\setbox1=\vbox{\hrule\box1}%
\advance\dimen1 by .4pt \ht1=\dimen1
\advance\dimen2 by .4pt \dp1=\dimen2 \box1\relax}

\begin{document}
	\title[Extreme values of the Dedekind Zeta function \\ on the critical line]{Extreme values of the Dedekind Zeta function \\ on the critical line}

\author[P. Nyadjo Fonga]{{Patrick Nyadjo Fonga}}
\address{Sorbonne Paris Nord university } 
\email{ patrick.nyadjofonga@edu.univ-paris13.fr }

	\maketitle

 \thispagestyle{empty}
 
	\begin{abstract}
	By employing the assessment of the asymptotic size of various sums of Gál studied by La Bretèche and Tenenbaum, we provide an improvement on the recent result of A. Bondarenko, P. Darbar, M. V. Hagen, W. Heap, and K. Seip regarding the large values of the Dedekind zeta-function on the critical line. Specifically, let $d\geqslant 3$ be an integer and $A$ be a positive constant. Denoting $K=\mathbb{Q}(\zeta_d)$, we establish that, if $T$ is sufficiently large, then uniformly for $d \ll (\log\log T)^A$ ,

\begin{equation*}
\max_{ t \in [0,T]}\left|\zeta_K \left(\frac{1}{2}+it \right) \right| \gg \exp\left({(1+o(1))\varphi(d)} \sqrt{\frac{\log T \log \log \log T}{\log \log T}} \right).
\end{equation*}
	\end{abstract}

\section{introduction}
Significant progress has been made in the study of extreme values of the Riemann zeta function on the critical line over the past five years. The most recent result improving that of Bondarenko and Seip \cite{BS1}, established by La Bretèche and Tenenbaum \cite{dlBT} states that; for a sufficiently large $T$, we have
\begin{equation} \label{eqdlBT}
\max_{t\in[0,T]} \left|\zeta \left(\frac{1}{2}+it \right) \right|
	\geqs \mathcal{L}(T)^{\sqrt{2}+o(1)},
\end{equation}
where the notation is defined as
\begin{equation*}
\mathcal{L}(x):=\exp \left\{\sqrt{\frac{\log x \log _3 x}{\log _2 x}} \right\} \quad(x>16),
\end{equation*}
and $\log_k(x)$ represents the logarithm function iterated $k$ times on $x$.
However, for $K=\mathbb{Q}\left(\zeta_d\right)$ (where $\zeta_d$ is a primitive $d$-th root of unity with $d$ being an integer greater than $3$), its close connection with the Dedekind zeta function, given by the relation
\begin{equation}\label{eq1}
\zeta_{K}(s)=\zeta(s) \prod_{\chi \neq \chi_0(\bmod d)} L\left(s, \chi^{*}\right),
\end{equation}
has motivated the recent study conducted by A. Bondarenko, P. Darbar, M. V. Hagen, W. Heap, and K. Seip \cite{A} on the large values of the Dedekind zeta function on the critical line. Here, $\chi_0$ represents the principal Dirichlet character modulo $d$, and the product is taken over all non-principal Dirichlet characters $\chi$ modulo $d$. The character $\chi^*$ is defined as the character that induces $\chi$ if $\chi$ is not primitive and $\chi^* = \chi$ otherwise. 
\newpage
They demonstrated that for an arbitrary positive constant $A$, if $T$ is sufficiently large, then uniformly for $d \ll (\log_2 T)^A$, there exists $t \in [0,T]$ such that
\begin{equation*}
\left|\zeta_K \left(\frac{1}{2}+it \right) \right| \gg \mathcal{L}(T)^{(1+o(1))\sqrt{\varphi(d)}}.
\end{equation*}
In this paper, we establish the following theorem, which improves this result by a factor of $\sqrt{\varphi (d)}$ in the exponent of the right-hand side:

\begin{theorem}\label{main thm}
Let $d \geqslant 3$, $T > 16$, and $K=\mathbb{Q}\left(\zeta_d\right)$. Let $A$ an arbitrary positive number. If $T$ is sufficiently large, then uniformly for $d \ll (\log_2 T)^A$, we have
\begin{equation*} \label{eqq1}
\max_{ t \in [0,T]}\left|\zeta_K \left(\frac{1}{2}+it \right) \right| \gg \mathcal{L}(T)^{(1+o(1))\varphi(d)}.
\end{equation*}
\end{theorem}
This result provides a more precise formulation of the dichotomy presented in \cite{A} and simultaneously improves (\ref{eqdlBT}) under a substantial condition that we specify in the following corollary:
%\vspace{0.5cm}
\begin{corollary} \label{main cor}
Let $A$ an arbitrary positive number. For $T$ sufficiently large, if there exists $d \ll (\log_2 T)^A$ such that for every non-principal Dirichlet character $\chi \pmod d$,
$$
\max _{t \in[0, T]}\left|L\left(\frac{1}{2}+i t, \chi\right)\right| \ll  \mathcal{L}(T)^{1/2+o(1)},
 \quad \text{then} \quad \max _{t \in[0, T]}\left|\zeta\left( \frac{1}{2}+it \right)\right|  \gg  \mathcal{L}(T)^{\left(1/2+o(1) \right) \varphi(d)}.$$  
\end{corollary}

\section{Background on the Dedekind zeta function} \label{sec2}
We recall the definition of the Dedekind zeta function of the number field $K= \mathbb{Q}(\zeta_d)$;
$$
\zeta_K(s)=\sum_{n=1}^{+\infty} \frac{a_K(n)}{n^s}=\sum_{I \neq\{0\} \text { ideal of } \mathcal{O}_K} \frac{1}{N(I)^s},
$$
where $a_K(n)$ denotes the number of ideals of the ring of integers $\mathcal{O}_K$ with norm $n$ and $\mathcal{R}e(s) > 1$.
From \cite{Wash}, we have the following theory:

Let $p$ be a prime number. It is known that $p$ is ramified in $K$ if and only if $p \mid d$, and otherwise, if we denote $f$ as the multiplicative order of $p$ in $(\mathbb{Z} / d \mathbb{Z})^{*}$, then $p$ decomposes into $r=\varphi(d)/f$ distinct prime ideals of norm $p^f$.
%Since these are the only prime ideals with a norm that is a power of $p$,
It follows that
\begin{equation} \label{eq2}
a_K\left(p^k\right)=\left\{\begin{array}{c}
N\left(\frac{k}{f}, r\right) \ \operatorname{if} \ f \mid k \\
0 \ \text { otherwise, }
\end{array}\right.
\end{equation}
where $N\left(\frac{k}{f}, r\right)$ denotes the number of ways to write $k/f$ as a sum of $r$ non-negative integers.\\
To simplify the notation, we use $a$ to denote $a_K$ throughout the rest of the paper. We also introduce the multiplicative function $$a'(n)= \prod_{p|n} \left(\frac{\varphi(d)+1}{2} \right),$$ with $a'(1)=1$.
%In particular, $a(p) = \varphi(d)$ for $p \equiv 1 \mod d$.
The relation (\ref{eq2}) directly implies the following lemma.
\begin{lemma} \label{lemma1}
We have:
\begin{enumerate}
    \item For a prime number $p$ such that $p \equiv 1 \mod d$, we have $$a(p)=  \varphi(d) \text{ and } a(p^2)= a(p)a'(p)=\varphi(d)( \varphi(d)+1)/2.$$
    \item If $m$ and $n$ are products of prime factors without square factors, all congruent to 1 modulo $d$, then $$a(mn) \geqslant a'(m)a(n).$$
\end{enumerate}
\end{lemma}
\section{Extreme value of G\'{a}l sums} \label{sec3}
Here, we utilize the methodology introduced in \cite{BS1} that involves G\'{a}l sums to explore the large values of the Riemann zeta function on the critical line. More specifically, our attention is directed towards the approach presented in \cite{dlBT} that we adapt for analyzing the case of the Dedekind zeta function. It is worth noting that, in a less general context, the G\'{a}l sums are defined as
\begin{equation} \label{eq3}
S_\alpha(\mathcal{M}):=\sum_{m, n \in \mathcal{M}} \frac{(m, n)^\alpha}{[m, n]^\alpha},
\end{equation}
where $\mathcal{M}$ is a finite set of positive integers and $\alpha$ is a positive real parameter. The main idea of \cite{dlBT} is to provide lower bounds for the large values of the Riemann zeta function on the critical line by using a carefully constructed G\'{a}l sum that also yields large values. More specifically, the work of \cite{dlBT} focuses on the G\'{a}l sum $S_{1/2}(\mathcal{M})$ when $\mathcal{M}$ is a specially constructed set. However, due to the non-trivial coefficients of the Dirichlet series associated with the Dedekind zeta function, we are led to study this new variant of the G\'{a}l sums
$$
S_{1/2}(\mathcal{M}, a):=\sum_{m, n \in \mathcal{M}} a\left(\frac{m}{(m,n)} \right)a \left(\frac{n}{(m,n)}\right)\sqrt{\frac{(m, n)}{[m, n]}},
$$
where $a(n)$ denotes the $n$-th Dirichlet coefficient of the Dedekind zeta function, as mentioned in Section \ref{sec2}.
We now adapt the construction of the set $\mathcal{M}$ from \cite{dlBT} to our specific situation. 
\subsection{Construction of the set $\mathcal{M}$} \label{sec3.1}
We introduce the parameters $u \in ]1, \mathrm{e}]$, $b \in ]1,+\infty[$, and $\gamma \in ]0,1[$. Then, we consider a sufficiently large integer $N$. Recall that the objective is to construct a set $\mathcal{M}$ of size $N$ for which the term $S_{1/2}(\mathcal{M},a)$ attains large values. For this purpose, we adopt the construction proposed in \cite{dlBT} as mentioned. We introduce the set:
$$I_k:= \left\{ p \in \left ]  u^k \log N \log _2 N, u^{k+1} \log N \log _2 N \right], p\equiv 1 \mod d \right\} ,\quad 1 \leqslant k \leqslant\left(\log _2 N\right)^{\gamma\varphi(d)}. $$
However, unlike the approach adopted by \cite{dlBT}, where only prime numbers in the intervals $\left]u^k \log N \log_2 N, u^{k+1} \log N \log_2 N\right]$ were considered, here we consider prime numbers $p$ such that $p \equiv 1 \mod d$. This is because they are prime numbers where we can have an explicit value of $a$ as shown by the relation (\ref{eq2}). To compensate for this restriction, we expand the number of intervals $I_k$ by choosing $k \in [1, (\log_2 N)^{\gamma\varphi(d)}]$ instead of $[1, (\log_2 N)^\gamma]$, as done in~\cite{dlBT}. \\
We define
		$${ P_k:=\left|I_k \right|=\pi\left(u^{k+1} \log N \log _2 N, d,1 \right)-\pi\left(u^k \log N \log _2 N, d, 1 \right) },$$
 where we have defined $$\pi(x,d,1) := \sum_{\substack{ p \leqslant x, \\ p \ \equiv   1 \bmod d}} 1.$$
From Siegel–Walfisz theorem, if $d \ll (\log (x))^A$, where $A$ is a positive real number, then $$\pi(x,d,1) \sim \frac{1}{\varphi(d)} \pi(x).$$
This gives for $ d \ll (\log_2(N))^A$ the equivalence
\begin{equation}
   {P_k \sim \frac{1}{\varphi(d)}u^k(u-1) \log N\left\{1+O\left(\frac{k+\log _3 N}{\log _2 N}\right)\right\}}.
\end{equation}
Next, we consider
		\begin{equation*}
		    {J_k}:=2\left\lfloor\frac{b \log N}{2 k^2 {\varphi(d)}  \log _3 N}\right\rfloor , \quad\left(1 \leqslant k \leqslant\left(\log _2 N\right)^{\gamma {\varphi(d)}}\right),
		\end{equation*} 
and we choose $b>1$ such that $b \gamma<1 / \log u$. \\
We denote $N_k:=\prod_{p \in {I_k}} p$, and then we consider the sets
		$$
		\mathcal{M}_k:=\left\{m: m=\frac{\ell }{q } {N_k}, \quad  \omega(\ell), \omega(q) \leqslant \frac{1}{2} J_k, \quad  \ell q \mid {N_k}\right\},
		$$ 
	and
		$$
		\mathcal{M}:=\left\{m=\prod_{1 \leqslant k \leqslant\left(\log _2 N\right)^{\gamma {\varphi(d)}}} m_k: \quad m_k \in \mathcal{M}_k \quad\left(1 \leqslant k \leqslant\left(\log _2 N\right)^{\gamma {\varphi(d)}}\right)\right\}.
		$$
Under the condition $d \ll \log_2 (N)^A$, following Lemma 2 of \cite{dlBT}, we have $|\mathcal{M}| \leqslant N$.
\subsection{Extreme value of $S_{1/2}(\mathcal{M}, a)$}
Here, we utilize the construction of the previous set $\mathcal{M}$ and provide a lower bound estimate for the term $S_{1/2}(\mathcal{M},a)/|\mathcal{M}|$. Observing that
$$S_{1/2}(\mathcal{M},a)=\prod_{1 \leqslant k \leqslant\left(\log _2 N\right)^{\gamma {\varphi(d)}}} S_{1/2}(\mathcal{M}_k,a),$$ we then focus first on the sum $S_{1/2}(\mathcal{M}_k,a)$ for a given $1 \leqslant k \leqslant\left(\log _2 N\right)^{\gamma {\varphi(d)}}$.

\begin{lemma} \label{lemma2}
Let's consider two elements in $\mathcal{M}_k$,
$
m:=\ell/qN_k \text{ and } m':=\ell'/q'N_k,
$
then we have
$$a\left(\frac{m}{\left(m, m'\right)}\right) a\left(\frac{m'}{\left(m, m'\right)}\right) \geqslant \frac{a'(\ell)a'(\ell')a(q)a(q')}{a'\left((\ell, \ell')\right)^2a\left((q, q')\right)^2}.$$
\end{lemma}
 \begin{proof}
 We have$$
		\left(m, m^{\prime}\right)=\frac{{N_k}}{\left[q, q^{\prime}\right]}\left(\frac{\ell q^{\prime}}{\left(q, q^{\prime}\right)}, \frac{\ell^{\prime} q}{\left(q, q^{\prime}\right)}\right)=\frac{{N_k}\left(\ell, \ell^{\prime}\right)}{\left[q, q^{\prime}\right]},
		$$
	and
		$$
		\frac{m}{\left(m, m^{\prime}\right)}=\frac{\ell}{\left(\ell, \ell^{\prime}\right)} \frac{q^{\prime}}{\left(q, q^{\prime}\right)}, \quad \frac{m^{\prime}}{\left(m, m^{\prime}\right)}=\frac{\ell^{\prime}}{\left(\ell, \ell^{\prime}\right)} \frac{q}{\left(q, q^{\prime}\right)}.
		$$
    Therefore $$a\left( \frac{m}{\left(m, m^{\prime}\right)} \right) a\left( \frac{m'}{\left(m, m^{\prime}\right)} \right) = a\left( \frac{\ell^{\prime}}{\left(\ell, \ell^{\prime}\right)} \frac{q}{\left(q, q^{\prime}\right)} \right) a\left( \frac{\ell}{\left(\ell, \ell^{\prime}\right)} \frac{q'}{\left(q, q^{\prime}\right)} \right).$$ Using Lemma \ref{lemma1} and the fact that  $$ a'\left( \frac{\ell^{\prime}}{\left(\ell, \ell^{\prime}\right)} \right)  = \frac{a'(\ell')}{a'\left(\ell, \ell^{\prime}\right)} \text{  and  } a\left( \frac{q}{\left(q, q^{\prime}\right)} \right)  = \frac{a(q)}{a\left(q, q^{\prime}\right)},  $$ then the proof is complete. 
 \end{proof}
From Lemma \ref{lemma2}, we have
		\begin{equation} \label{eq5}
 S_{1/2}(\mathcal{M}_k,a) \geqslant \sum_{\substack{\ell, \ell^{\prime} \mid {N_k} \\ \omega(\ell), \omega\left(\ell^{\prime}\right) \leqslant J_k / 2}} \frac{\left(\ell, \ell^{\prime}\right)a'(\ell)a'(\ell')}{ a'( (\ell,\ell') ) ^2 \sqrt{\ell \ell^{\prime}}} \sum_{\substack{q, q^{\prime} \mid {N_k} \\\left(q, \ell\right)=\left(q^{\prime}, \ell'\right)=1 \\ \omega(q), \omega\left(q^{\prime}\right) \leqslant J_k / 2}} \frac{\left(q, q^{\prime}\right) a(q)a(q')}{ a( (q,q') )^2 \sqrt{q q^{\prime}}}.
\end{equation}
		Let $F_{k,a}\left(\ell, \ell'\right)$ denote the inner sum. The identity $n=\sum_{d_1 \mid n} \varphi(d_1)$ allows to write that
		$$
		F_{k,a}\left(\ell, \ell'\right) = \sum_{\substack{d_1 \mid {N_k} \\
				\omega\left(d_1\right) \leqslant J_k / 2 \\
				\left(d_1, \ell \ell^{\prime}\right)=1}} \frac{\varphi\left(d_1\right)}{d_1 } \sum_{\substack{n_1, n_1^{\prime} \mid {N_k} \\
				\left(n_1, \ell d_1\right)=\left(n_1^{\prime}, \ell^{\prime} d_1\right)=1 \\
				\omega\left(n_1\right), \omega\left(n_1^{\prime}\right) \leqslant {J_k} / 2-\omega\left(d_1\right)}} \frac{ a(n_1)a(n'_1)}{ a^2((n_1,n'_1)) \sqrt{n_1 n_1^{\prime}} }.$$ 
We now restrict the inner sum to the case where $(n_1,n'_1)=1$,
so that 
		$$ F_{k,a}\left(\ell, \ell'\right) \geqslant \sum_{\substack{d_1 \mid {N_k} \\
				\omega\left(d_1\right) \leqslant {J_k} / 2 \\
				\left(d_1, \ell \ell^{\prime}\right)=1}} \frac{\varphi\left(d_1\right)}{d_1} \sum_{\substack{n'_1 \mid N_k \\ \left(n_1^{\prime}, \ell^{\prime} d_1\right)=1 \\ \omega\left(n_1^{\prime}\right) \leqslant {J_k} / 2-\omega\left(d_1\right)}} \frac{a(n'_1)}{  \sqrt{n_1^{\prime}} }\sigma\left(a,\frac{1}{2} {J_k}-\omega\left(d_1\right), \ell d_1n'_1\right),
		$$
where $$ \sigma(a,R,r)= \sum_{\substack{n \mid {N_k} \\
				\omega\left(n\right) \leqslant R \\
				\left(n,r\right)=1}} \frac{a(n)}{\sqrt{n}}. $$ 
Since
		$$
		\frac{\varphi(d_1)}{d_1} \gg \exp \left\{-\sum_{p \in {I_k}} \frac{1}{p}\right\} \gg 1 \quad\left(d_1 \mid {N_k} \right),
		$$
		it follows that
		$$
		F_{k,a}\left(\ell, \ell'\right) \gg \sum_{\substack{d_1 \mid {N_k} \\ \omega\left(d_1\right) \leqslant {J_k} / 2 \\\left(d_1, \ell \ell^{\prime}\right)=1}} \sum_{\substack{n'_1 \mid N_k \\ \left(n_1^{\prime}, \ell^{\prime} d_1\right)=1 \\ \omega\left(n_1^{\prime}\right) \leqslant {J_k} / 2-\omega\left(d_1\right)}} \frac{a(n'_1)}{  \sqrt{n_1^{\prime}} }\sigma\left(a,\frac{1}{2} {J_k}-\omega\left(d_1\right), \ell d_1n'_1\right).
		$$
		By performing the same calculation with the outer sum of the relation (\ref{eq5}), we obtain
$$ S_{1/2}(\mathcal{M}_k,a) \gg \sum_{\substack{d_1, d_2 \mid {N_k} \\ \omega\left(d_j\right) \leqslant {J_k} / 2 \\\left(d_1, d_2 \right)=1 }} \sum_{\substack{n_2, n'_2 \mid {N_k} \\\left(n_2n'_2, d_1d_2 \right)=1  \\ \omega\left(n_2\right) \leqslant {J_k} / 2 - \omega\left(d_2\right) \\ \omega\left(n'_2\right) \leqslant {J_k} / 2 - \omega\left(d_2\right) \\ (n_2, n'_2)=1 }}  \frac{ a'(n_2)a'(n'_2)}{\sqrt{n'_2n_2}} \sum_{\substack{n'_1 \mid N_k \\ \left(n_1^{\prime}, d_2n_2^{\prime} d_1\right)=1 \\ \omega\left(n_1^{\prime}\right) \leqslant {J_k} / 2-\omega\left(d_1\right)}} \frac{a(n'_1)}{  \sqrt{n_1^{\prime}} }\sigma\left(a,\frac{1}{2} {J_k}-\omega\left(d_1\right), d_2n_2 d_1n'_1\right).$$
We introduce
		$$
		j_k:=\left\lfloor\frac{\lambda}{k} \sqrt{\frac{\log N}{\log _2 N \log _3 N}}\right\rfloor
		$$
where $\lambda$ is a bounded parameter, and we restrict the previous outer sum to pairs $\left(d_1, d_2\right)$ such that $\omega\left(d_j\right) \leqslant \frac{1}{2} J_k-j_k$, and the inner sum to integers $n_2, n_2^{\prime}$ such that $\omega\left(n_2\right)=\omega\left(n_2^{\prime}\right)=j_k$. Together with the relation
		\begin{equation*}
		\sigma\left(a,\frac{1}{2} {J_k}-\omega\left(d_1\right), n_2d_2d_1n'_1\right) \geqslant \sigma\left(a,j_k, n_2d_2d_1n'_1\right),
  \end{equation*}
  we get 
$$ S_{1/2}(\mathcal{M}_k,a) \gg \sum_{\substack{d_1, d_2 \mid {N_k} \\ \omega\left(d_j\right) \leqslant {J_k} / 2-j_k \\\left(d_1, d_2 \right)=1 }} \sum_{\substack{n_2, n'_2 \mid {N_k} \\\left(n_2n'_2, d_1d_2 \right)=1  \\ \omega\left(n_2\right), \omega\left(n'_2\right)= j_k \\ (n_2, n'_2)=1 }}  \frac{ a'(n_2)a'(n'_2)}{\sqrt{n'_2n_2}} \sum_{\substack{n'_1 \mid N_k \\ \left(n_1^{\prime}, d_2n_2^{\prime} d_1\right)=1 \\ \omega\left(n_1^{\prime}\right) \leqslant {J_k} / 2-\omega\left(d_1\right)}} \frac{a(n'_1)}{  \sqrt{n_1^{\prime}} }\sigma\left(a,j_k, n_2d_2d_1n'_1\right).$$
We set $$
		T=\frac{2 k \mathrm{e}(\sqrt{u}-1) u^{k / 2}}{\alpha} \sqrt{\log _3 N}\left\{1+O\left(\frac{\sqrt{\log _2 N \log _3 N}}{\sqrt{\log N}}\right)\right\},$$ so that according to the relation (2.10) of \cite{dlBT}, we have  \begin{equation*} \label{(2)}
		    \sigma\left(a,j_k, n_2n'_1 d_2 d_1\right) \geqslant  T^{j_k} \mathrm{e}^{o\left(j_k\right)}.
		\end{equation*}
Therefore,
		$$ S_{1/2}(\mathcal{M}_k,a) \gg T^{j_k} \mathrm{e}^{o\left(j_k\right)} \sum_{\substack{d_1, d_2 \mid {N_k} \\ \omega\left(d_j\right) \leqslant {J_k} / 2-j_k \\\left(d_1, d_2 \right)=1 }} \sum_{\substack{n_2, n'_2 \mid {N_k} \\\left(n_2n'_2, d_1d_2 \right)=1  \\ \omega\left(n_2\right), \omega\left(n'_2\right) \leqslant {J_k} / 2 - \omega\left(d_2\right) \\ (n_2, n'_2)=1 }}  \frac{ a'(n_2)a'(n'_2)}{\sqrt{n'_2n_2}} \sum_{\substack{n'_1 \mid N_k \\ \left(n_1^{\prime}, d_2n_2^{\prime} d_1\right)=1 \\ \omega\left(n_1^{\prime}\right) \leqslant {J_k} / 2-\omega\left(d_1\right)}} \frac{a(n'_1)}{  \sqrt{n_1^{\prime}} }.$$
Using a similar calculation, we have that $$\sum_{\substack{n'_1 \mid N_k \\ \left(n_1^{\prime}, d_2n_2^{\prime} d_1\right)=1 \\ \omega\left(n_1^{\prime}\right) \leqslant {J_k} / 2-\omega\left(d_1\right)}} \frac{a(n'_1)}{  \sqrt{n_1^{\prime}} } =  \sigma\left(a,{J_k} / 2-\omega\left(d_1\right), n'_2 d_2 d_1\right) \gg T^{j_k} \mathrm{e}^{o\left(j_k\right)},$$
thus 
  $$ S_{1/2}(\mathcal{M}_k,a) \gg T^{2j_k} \mathrm{e}^{o\left(j_k\right)} \sum_{\substack{d_1, d_2 \mid {N_k} \\ \omega\left(d_j\right) \leqslant {J_k} / 2 \\\left(d_1, d_2 \right)=1 }} \sum_{\substack{n_2, n'_2 \mid {N_k} \\\left(n_2n'_2, d_1d_2 \right)=1  \\ \omega\left(n_2\right), \omega\left(n'_2\right) \leqslant {J_k} / 2 - \omega\left(d_2\right) \\ (n_2, n'_2)=1 }}  \frac{ a'(n_2)a'(n'_2)}{\sqrt{n'_2n_2}} .$$
		By a similar calculation, we obtain
		$$
	 \sum_{\substack{n_2, n'_2 \mid {N_k} \\\left(n_2n'_2, d_1d_2 \right)=1  \\ \omega\left(n_2\right), \omega\left(n'_2\right) \leqslant {J_k} / 2 - \omega\left(d_2\right) \\ (n_2, n'_2)=1 }}  \frac{ a'(n_2)a'(n'_2)}{\sqrt{n'_2n_2}}\gg \left(\frac{T}{2} \right)^{2 j_k} \mathrm{e}^{o\left(j_k\right)}.
		$$
		Therefore,
		\begin{equation} \label{(3)}
		S_{1/2}(\mathcal{M}_k,a) \gg \left({\frac{T}{\sqrt{2}}}\right)^{4 j_k} \mathrm{e}^{o\left(j_k\right)} \sum_{\substack{d_1, d_2 \mid {N_k} \\
				\omega\left(d_j\right) \leqslant {J_k} / 2-j_k \\
				\left(d_1, d_2, \right)=1}} 1.
		\end{equation}
The relation (2.13) of \cite{dlBT} shows that $$ \sum_{\substack{d_1, d_2 \mid {N_k} \\
				\omega\left(d_j\right) \leqslant {J_k} / 2-j_k \\
				\left(d_1, d_2, \right)=1}} 1 \geqslant\left|\mathcal{M}_k\right|\left(\frac{b+o(1)}{2 k^2 u^k(u-1) \log _3 N}\right)^{2 j_k}, $$
so that $$
\frac{S_{1/2}(\mathcal{M}_k,a)}{\left|\mathcal{M}_k\right|} \gg\left(\frac{h}{\lambda^2}\right)^{2 j_k} \mathrm{e}^{o\left(j_k\right)},
$$
with $h:= \mathrm{e}^2 b(\sqrt{u}-1) /(\sqrt{u}+1)$.
Taking the product over $k$ from $1$ to $(\log_2N)^{\gamma \varphi(d)}$, we obtain: \begin{equation} \label{main rel}
 \frac{S_{1/2}(\mathcal{M},a)}{\left|\mathcal{M}\right|} \gg \mathcal{L}(N)^{\beta + o(1)},
 \end{equation}
with $\beta:=2 \gamma {\varphi(d)} \lambda \log \left(h / \lambda^2\right)$. The optimal choice $\lambda:=\sqrt{h} / \rm{e} $  gives
		$$
		\beta=\frac{4 \gamma{\varphi(d)} \sqrt{h}}{\mathrm{e}}=4 \gamma {\varphi(d)} \sqrt{\frac{ b(\sqrt{u}-1)}{\sqrt{u}+1}}= 4 \gamma {\varphi(d)}\sqrt{\frac{ b(\sqrt{u}-1)}{\sqrt{u}+1}}  .
		$$
		By choosing $b \gamma^2 \log u$ close to $1$ and then choosing $u$ close to $1$, we have $\sqrt{u}-1$ approaching $(\log u) /2$, and $\sqrt{u}+1$ approaching $2$. Consequently, $\beta$ converges to $2 \varphi(d)$.
\section{Large values of the Dedekind zeta function.}
We begin by recalling the definition of the Fourier transform of a function $f \in L^1(\mathbb{R})$, which is given by:
$$
\hat{f}(\xi):=\int_{\mathbb{R}} f(x) \mathrm{e}^{-i x \xi} \mathrm{d} x \quad(\xi \in \mathbb{R}) .
$$
Next, we present two lemmas.
\begin{lemma} \label{lemma3} \cite{dlBT}
 Let $\sigma \in [\frac{1}{2}, 1[$ and $F$ be a holomorphic function in the horizontal strip $y=\Im(z) \in [\sigma-2,0]$, satisfying the growth condition
$$
\sup _{\sigma-2 \leqslant y \leqslant 0}|F(z)| \leq \frac{1}{x^2+1},
$$
then, for all $s=\sigma+i t \in \mathbb{C}, t \neq 0$, we have
\begin{align*}
\int_{\mathbb{R}} \zeta_{K}(s+i u) \zeta_{K} & ( \bar{s} + iu) F(u) {\rm d} u
 \\=& \sum_{k,\ell >1} \frac{\hat{F}(\log k \ell) a(k)a(\ell)}{k^{s} \ell^{\bar{s}}} -2 \pi \zeta_{K}(1-2 i t) F(i s-i) -2 \pi \zeta_{K}(1+2 i t) F(i \bar{s} - i).
 \end{align*}
\end{lemma}
Let $T>1$ be a fixed value. Following \cite{A}, we consider the function.
$$
K_{\mathrm{\eta}}(u):=\frac{\sin ^{2 \mathrm{\eta}}\left(\left(\frac{1}{\mathrm{\eta}} \varepsilon \log T\right) u\right)}{\left(\frac{1}{\mathrm{\eta}} \varepsilon \log T\right)^{2 \mathrm{\eta}-1} u^{2 \mathrm{\eta}}},
$$
with $\varepsilon$ small and $\mathfrak{\eta}$ large in $\mathbb{N}$ to be chosen. 
 \begin{lemma} \cite{A}
Let $K_{\mathrm{\eta}}(u)$ be as above. Then $\widehat{K_{\mathrm{\eta}}}(v)$ is a real and even function, satisfying $0 \leqslant \widehat{K_{\mathrm{\eta}}}(v) \leqslant \widehat{K}_{\mathrm{\eta}}(0)$ and decreasing on $[0, \infty[$ with
$$
\left|\frac{d}{d v} \widehat{K_{\mathrm{\eta}}}(v)\right| \leqslant \frac{\widehat{K_{\mathrm{\eta}-1}}(0)}{\frac{1}{\mathrm{\eta}} \varepsilon \log T}.
$$
Moreover, for large values of $\mathfrak{\eta}$, we have
$$
\widehat{K_{\mathrm{\eta}}}(0) \sim \sqrt{\frac{3 \pi}{\mathrm{\eta}}}.
$$
\end{lemma}
\subsection{ Proof of Theorem \ref{main thm}} 
We specialize Soundararajan’s resonance method \cite{Sound res} to the Dedekind zeta function. Let $\beta \in ]0, 1[$, and let $\kappa = 1 - \beta$. Throughout the following, we assume that $d \ll (\log_2 T^\kappa)^A$, where $A$ is the positive real number introduced in Section \ref{sec3.1}.\\
Let $\mathcal{M}$ be a set of integers with cardinality less than or equal to $N$. We define:
$$
\mathcal{M}_j:=\mathcal{M} \cap \left](1+1 / T)^j \cdot(1+1 / T)^{j+1}\right] \quad(j \geqslant 0),
$$
We also denote $h_j:=\min \mathcal{M}_j$ when $\mathcal{M}_j \neq \emptyset$, and we define $\mathcal{H}$ as the set of all $h_j$. We then define $r: \mathcal{H} \rightarrow \mathbb{R}^{+}$ by the formula:$$r\left(h_j\right)^2:=\sum_{m \in \mathcal{M}_j} 1.$$
We consider the resonance factor in the form of $|R(t)|^2$ with:
$$
R(t)=\sum_{h \in \mathcal{M}} \frac{r(h)}{h^{i t}}.
$$
We also consider the Gaussian density $\Phi(t) := e^{-t^2 / 2}$, so that $\hat{\Phi}(y) = \sqrt{2 \pi} \Phi(y)$.
We introduce the following notations: 
$$
\begin{aligned}
\xi (t, u) & :=\zeta_{K}\left(\frac{1}{2}+i t+i u\right) \zeta_{K} \left(\frac{1}{2}-i t+i u\right) K_\eta (u), \\
I(T) & :=\int_{\mathbb{R}} \int_{\mathbb{R}}|R(t)|^2 \Phi\left(\frac{t \log T}{T}\right)  \xi (t, u) {\rm d} u {\rm d} t .
\end{aligned}
$$
These notations are similar to those used in \cite{dlBT}, except for the regularization factor $K_\eta$, which differs from theirs and corresponds to the one considered in \cite{A}. Indeed, the importance of this function lies precisely in controlling the large values of the Dedekind zeta function, which can potentially be higher than those of the Riemann zeta function studied in
\cite{dlBT}.\par
We denote $$Z_\beta(T) = \max_{ T^\beta \leq t \leq T } \left|\zeta_{K} \left(\frac{1}{2}+ it \right)\right|.$$
According to Lemma 5 of \cite{BS1},
  \begin{equation} \label{eq6}
\int_{\mathbb{R}}|R(t)|^2 \Phi \left(\frac{t \log T}{T}\right) \mathrm{d} t \ll \frac{T|\mathcal{M}|}{\log T}.
 \end{equation}
\begin{lemma} \label{lemma5}
For large $T$, we have for $\eta = 2 \varphi(d)$ and $N = \lfloor T^\kappa \rfloor$ that \begin{equation*}
I(T) + O( {|\mathcal{M}|} T \log^2 T)= \int_{2 T^\beta \leq |t| \leq T / 2}|R(t)|^2 \Phi\left(\frac{t \log T}{T}\right) \int_{\mathbb{R}} \xi(t, u) \mathrm{d} u \mathrm{~d}t.
\end{equation*}
\end{lemma}
\begin{proof}
We are first interested in
\begin{equation} \label{(5)}
 \int_{|t| \leqslant 2 T^\beta} \int_{|u|>  T^\beta} \zeta_{K}\left(\frac{1}{2}+i t+i u\right) \zeta_{K} \left(\frac{1}{2}-i t+i u\right) K_\eta(u) {\rm d}t {\rm d}u .
 \end{equation}
We begin by using the upper bound for the Dedekind zeta function as given in \cite{A},$$ \left| \zeta_{K}\left(\frac{1}{2}+i t+i u\right) \zeta_{K} \left(\frac{1}{2}-i t+i u\right) \right| \ll \left( cd(1+|t|+|u|) \right)^{ \varphi(d)/2}.$$
From the definition of $K_\eta$, since $ d \ll (\log_2T)^A$, we have $$ |K_\eta(u)| \ll \frac{1}{[d(2|u|+1)]^{\eta}}. $$
The inner integral of (\ref{(5)}) then becomes: $$\ll  \int_{|u|>  T^\beta} \frac{\left( cd(1+|t|+|u|) \right)^{ \varphi(d)/2}}{[d(1+|t|+|u|)]^{2\varphi(d)}} {\rm d} u \ll 1.$$ 
By completing with the outer integral, we have that: $$ \int_{|t| \leqslant 2 T^\beta} \int_{|u|>  T^\beta} \zeta_{K}\left(\frac{1}{2}+i t+i u\right) \zeta_{K} \left(\frac{1}{2}-i t+i u\right) K_\eta(u) {\rm d} t {\rm d} u \ll T^\beta. $$
Now, let's focus on the term 
$$ \int_{|t| \leqslant 2 T^\beta} \int_{|u| \leqslant  T^\beta} \zeta_{K}\left(\frac{1}{2}+i t+i u\right) \zeta_{K} \left(\frac{1}{2}-i t+i u\right) K_\eta (u) {\rm d}t {\rm d}u.$$
We use $2|a b| \leqslant|a|^2+|b|^2$ so that 
\begin{align*}
	 \int_{|t| \leqslant 2 T^\beta} \int_{|u| \leq  T^\beta } \xi(t, u) {\rm d} u {\rm d} t & \ll \int_{|\mathrm{t}| \leqslant 2 T^\beta} \int_{|\mathrm{u}| \leqslant T^\beta}\left|\zeta_{K} \left(\frac{1}{2}+i t+i u\right)\right|^2 K_\eta(u) {\rm d} u {\rm d} t \\ & \ll \widehat{K_{\mathrm{\eta}}}(0) \int_{|\mathrm{t}| \leqslant 3T^\beta}\left|\zeta_{K} \left(\frac{1}{2}+i t\right)\right|^2 {\rm d} t.
\end{align*}
Next, we use the second-order moment bound of the Dedekind zeta function from \cite{WH} and we get
$$ \int_{|t| \leqslant 2 T^\beta} \int_{u \in \mathbb{R} } 3(t, u) {\rm d} u {\rm d} t \ll {\widehat{K_{\mathrm{\eta}}}(0)}T^\beta \log^2 T + T^\beta   \ll  T^\beta \log^2 T .$$
Since  $ \Phi\left(t \log T/{T}\right) \leqslant 1 $ and $ |R(t)|^2 \leqslant R(0)^2, $ it follows that
$$
\int_{|t| \leq 2 T^\beta}|R(t)|^2 \Phi\left(\frac{t \log T}{T}\right) \int_{\mathbb{R}} \xi(t, u) {\rm d} u {\rm d} t \ll  R(0)^2 T^\beta \log^2 T.
$$
However, the rapid decay of $\Phi$ allows to write
$$
\int_{|t|>T / 2}|R(t)|^2 \Phi\left(\frac{t \log T}{T}\right) \int_{\mathbb{R}} \xi(t, u) \mathrm{d} u \mathrm{~d} t \ll R(0)^2
.$$ 
Noting that $$R(0)^2 \leqslant N \sum_{h \in \mathcal{H}} r(h)^2 \leqslant N |\mathcal{M}|,$$ this completes the proof of the lemma.
\end{proof}
Throughout the rest of the paper, we assume definitely $\eta = 2\varphi(d)$, $N = \lfloor T^\kappa \rfloor$,  so that $d \ll (\log_2 N)^A$.
\begin{lemma} For a large $T$, we have
    $$I(T) + O({|\mathcal{M}|} T \log^2 T) \ll \frac{T {|\mathcal{M}|}}{\log T} Z_\beta(T)^2. $$
\end{lemma}
\begin{proof} We have
\begin{align*}
	 \int_{|u|> |t|/2} \xi(t, u) \mathrm{d} u &=  \int_{|u|> |t|/2} \zeta_{K}\left(\frac{1}{2}+i t+i u\right) \zeta_{\mathbb{K}} \left(\frac{1}{2}-i t+i u\right) K(u) {\rm d} u  \ll  1.
\end{align*}
It follows from (\ref{eq6}) that $$ \int_{2 T^\beta \leq |t| \leq T / 2}|R(t)|^2 \Phi\left(\frac{t \log T}{T}\right) \int_{|u|> |t|/2} \xi(t, u) \mathrm{d} u \mathrm{~d}t \ll \frac{T {|\mathcal{M}|}}{ \log T}.$$
Using Lemma \ref{lemma5}, we obtain
 $$I(T) + O({|\mathcal{M}|} T \log^2 T) = \int_{2 T^\beta \leq |t| \leq T / 2}|R(t)|^2 \Phi\left(\frac{t \log T}{T}\right) \int_{|u| \leq |t|/2} \xi(t, u) \mathrm{d} u \mathrm{~d}t.$$ 
Therefore, the proof of the lemma follows from the observation that 
$$ \int_{2 T^\beta \leq |t| \leq T / 2}|R(t)|^2 \Phi\left(\frac{t \log T}{T}\right) \int_{|u| \leq |t|/2} 3(t, u) \mathrm{d} u \mathrm{~d}t \ll \frac{T {|\mathcal{M}|}}{\log T} Z_\beta(T)^2.$$
\end{proof}
Now, we associate $I(T)$ with the Gál sum from Section \ref{sec3}. \\
Applying Lemma \ref{lemma3} with $F=K_\eta$ yields.
$$
I(T)=I_1(T)+I_2(T)+I_3(T),
$$
where we have defined
$$
\begin{aligned}
& I_1(T):=\int_{\mathbb{R}} \sum_{k,\ell >1} \frac{\hat{K_\eta}(\log k \ell) a(k)a(\ell)}{ \sqrt{k\ell } (k/ \ell )^{it}} |R(t)|^2 \Phi\left(\frac{t \log T}{T}\right) \mathrm{d} t, \\
& I_2(T):=-2 \pi \int_{\mathbb{R}} \zeta_{K}(1-2 i t) K_\eta\left(-t-\frac{1}{2} i\right)|R(t)|^2 \Phi\left(\frac{t \log T}{T}\right) \mathrm{d} t, \\
& I_3(T):=-2 \pi \int_{\mathbb{R}} \zeta_{K}(1+2 i t) K_\eta\left(t-\frac{1}{2} i\right)|R(t)|^2 \Phi\left(\frac{t \log T}{T}\right) \mathrm{d} t,
\end{aligned}
$$
so that by using $$\left|I_2(T)\right|+\left|I_3(T)\right| \ll T{|\mathcal{M}|} / \log T,$$
we obtain $$
\frac{T{|\mathcal{M}|}}{\log T} Z_\beta(T)^2 \gg I_1(T)+O({|\mathcal{M}|} T \log^2 T) .
$$
In the following, we consider the set $\mathcal{M}$ to be the one defined in Section \ref{sec3.1}. 
\begin{lemma} \label{tron} For a large $T$, we have
    $$ I_1(T) \gg \frac{T}{\log T} \sum_{\substack{m, n \in \mathcal{M} \\ [m, n] /(m, n) \leqslant T^{\varepsilon/ 3 \eta}}} a \left(\frac{m}{(m,n)}\right) a\left( \frac{n}{(m,n)}\right )\sqrt{\frac{(m, n)}{[m, n]}}. $$
\end{lemma}
\begin{proof} We have
$$
\begin{aligned}
I_1(T) & =\frac{T \sqrt{2 \pi}}{\log T} \sum_{h, h' \in \mathcal{H}} r(h) r(h') \sum_{k, \ell \geqslant 1} \frac{\hat{K_\eta}(\log k \ell) a(\ell)a(k)}{\sqrt{k \ell}} \Phi\left(\frac{T}{\log T} \log \frac{h k}{h' \ell}\right), \\
& \gg \frac{T}{\log T} \sum_{k, \ell \geqslant 1 } \frac{\hat{K_\eta}(\log k \ell)a(\ell)a(k)}{\sqrt{k \ell}} \sum_{h, h' \in \mathcal{H}} r(h) r(h') \Phi\left(\frac{ T}{\log T} \log \frac{h k}{h' \ell}\right).
\end{aligned}
$$
Following the same technique used in the proof of Lemma 6 in \cite{A}, we obtain 
$$ I_1(T) \gg \frac{{\widehat{K_{\mathrm{\eta}}}(0)}T}{\log T} \sum_{1 \leqslant k \ell \leq T^{\varepsilon/3\eta}} \frac{a(\ell)a(k)}{\sqrt{k \ell}} \sum_{h, h' \in \mathcal{H}} r(h) r(h') \Phi\left(\frac{ T}{\log T} \log \frac{h k}{h' \ell}\right).$$
When $k$ and $\ell$ are fixed and if $h \in \mathcal{M}_i$ and $h' \in \mathcal{M}_j$, we have
$$
\sum_{\substack{m \in M_i, n \in M_j \\ m k=n \ell}} 1 \leqslant \min \left\{r(h)^2, r(h')^2\right\} \leqslant r(h) r(h'),
$$
so that
$$
\sum_{\substack{m \in \mathcal{M}_i, n \in \mathcal{M}_j \\ m k=n \ell }} \Phi\left(\frac{T}{\log T} \log \frac{h k}{h' l}\right) \leqslant r(h) r(h') \Phi\left(\frac{T}{\log T} \log \frac{h k}{h' \ell}\right) .
$$
In the left-hand side, the argument of $\Phi$ is $\ll 1/\log T$. Therefore, by summing over $i$ and $j$, we obtain
$$
I_1(T) \gg \frac{T}{\log T} \sum_{\substack{m \in \mathcal{M}, n \in \mathcal{M} \\ m k=n \ell }} \sum_{\substack{1 \leqslant k \ell \leqslant T^{ \varepsilon / 3\eta}}} \frac{a(\ell)a(k)}{\sqrt{k \ell}}.
$$
For a pair $(m, n) \in \mathcal{M}^2$,
we restrict the inner sum to pairs $(k, \ell)$ such that
$$
k= \frac{m}{(m,n)}, \quad \ell=\frac{n }{(m,n)}, \quad k \ell=[m, n] /(m, n) \leqslant T^{ \varepsilon / 3\eta},
$$
so that
\begin{equation} \label{(6)}
I_1(T) \gg \frac{T}{\log T} \sum_{\substack{m, n \in \mathcal{M} \\ [m, n] /(m, n) \leqslant T^{\varepsilon/ 3\eta}}} a \left(\frac{m}{(m,n)}\right) a\left( \frac{n}{(m,n)}\right ) \sqrt{\frac{(m, n)}{[m, n]}} .
\end{equation}
\end{proof}
However, using the Rankin’s trick, we get
\begin{align*}
 & \sum_{\substack{m, n \in \mathcal{M} \\ [m, n] /(m, n) \leqslant T^{\varepsilon/ 3n}}} a \left(\frac{m}{(m,n)}\right) a\left( \frac{n}{(m,n)}\right )\sqrt{\frac{(m, n)}{[m, n]}}  \geqslant S_{1/2}(\mathcal{M},a)  - \frac{S_{1/3}(\mathcal{M},a)}{T^{\varepsilon / 18\eta}},
\end{align*} 
where $$S_{1/3}(\mathcal{M},a) = \sum_{\substack{m, n \in \mathcal{M} }} a \left(\frac{m}{(m,n)}\right) a\left( \frac{n}{(m,n)}\right ) \left(\frac{(m, n)}{[m, n]}\right)^{1/3}. $$
Therefore, from Lemma \ref{tron}, we have
\begin{equation} \label{(7)}
Z_\beta(T)^2 \gg \frac{S_{1/2}(\mathcal{M},a)}{|\mathcal{M}|}-\frac{S_{1/3}(\mathcal{M},a)}{|\mathcal{M}| T^{\varepsilon / 18 \eta}}+O\left((\log T)^3\right).
\end{equation}
We set $y_{\mathcal{M}}:=\max _{m \in \mathcal{M}} P^{+}(m)$ where $ P^{+}(m)$ is the largest prime factor of $m$, with the convention $P^{+}(1) = 1$,
 so that 
$$
\frac{S_{1/3}(\mathcal{M},a)}{|\mathcal{M}|} \ll \exp \left\{ \varphi(d)y_{\mathcal{M}}^{2 / 3}\right\}.
$$
Since $y_{\mathrm{M}} \ll \log T \log_2 T \ll(\log T)^{6 / 5}$ for sufficiently large $T$, we have  $$ \frac{S_{1/3}(\mathcal{M},a)}{T^{\varepsilon/18 \eta}|\mathcal{M}|} \ll \frac{\exp \left\{ \varphi(d)y_{\mathcal{M}}^{2 / 3}\right\}}{T^{\varepsilon/18 \eta}} \longrightarrow 0.$$ This implies that the negative term in (\ref{(7)}) is negligible.
Now, we replace $N=\left\lfloor T^{1-\beta}\right\rfloor$ so that based on relation (\ref{main rel}), we obtain
$$
Z_\beta(T) \gg \mathcal{L}(T)^{\varphi(d)\sqrt{1-\beta}+o(1)}.
$$
This completes the proof of Theorem \ref{main thm}.
The corollary \ref{main cor} follows trivially from this result.
\vspace{1cm}

\noindent\textbf{Acknowledgements:} The author extends their heartfelt appreciation to Professor Régis de la Bretèche for his meticulous guidance throughout every stage of this project. 
The   author is supported by funding from a scholarship of IMJ-PRG during his Master thesis.

\end{document}